\documentclass[11pt]{article}
\usepackage[T1]{fontenc}
\usepackage{lmodern}
\usepackage[a4paper,margin=29mm]{geometry}
\usepackage{amsmath,amssymb,amsthm,mathtools}
\usepackage{microtype}
\usepackage{enumitem}
\usepackage{needspace}
\newcommand{\rev}[1]{#1}
\newenvironment{revision}{\begingroup}{\par\endgroup}
\usepackage[title]{appendix}
\usepackage[hidelinks]{hyperref}
\usepackage{bookmark}
\usepackage{fancyhdr}
\fancypagestyle{plain}{\fancyhf{}\fancyfoot[C]{\thepage}}
\setlist[enumerate]{itemsep=3pt,topsep=5pt}
\setlist[itemize]{itemsep=3pt,topsep=5pt}
\allowdisplaybreaks[2]
\newtheorem{theorem}{Theorem}[section]
\newtheorem{proposition}[theorem]{Proposition}
\newtheorem{corollary}[theorem]{Corollary}
\newtheorem{lemma}[theorem]{Lemma}
\theoremstyle{definition}

\newtheorem{example}[theorem]{Example}

\theoremstyle{remark}
\newtheorem{remark}[theorem]{Remark}
\numberwithin{equation}{section}
\newcommand{\C}{\mathbb C}

\newcommand{\Z}{\mathbb Z}
\newcommand{\GL}{\operatorname{GL}}
\newcommand{\SL}{\operatorname{SL}}
\newcommand{\SU}{\operatorname{SU}}
\newcommand{\UU}{\operatorname{U}}
\newcommand{\Aut}{\operatorname{Aut}}
\newcommand{\Inn}{\operatorname{Inn}}

\newcommand{\Hom}{\operatorname{Hom}}
\newcommand{\Epi}{\operatorname{Epi}}

\newcommand{\tr}{\operatorname{tr}}
\newcommand{\rank}{\operatorname{rank}}

\newcommand{\diag}{\operatorname{diag}}

\newcommand{\sulie}{\mathfrak{su}}

\newcommand{\Zcl}[1]{\overline{#1}^{\,\mathrm{Zar}}}

\newcommand{\PP}{\operatorname{PU}}
\hypersetup{pdftitle={Cosets with constant characteristic polynomial},pdfauthor={Nir Avni and Tsachik Gelander},pdfsubject={Linear groups, constant characteristic polynomials, and free-group automorphisms}}

\title{\textbf{Cosets with constant characteristic polynomial}}
\author{Nir Avni \and Tsachik Gelander}
\date{\rev{September 28, 2026}}

\begin{document}
\maketitle
\vspace{-1.5em}
\begin{abstract}
Let $k$ be a field, $H\leq\GL_d(k)$, and $x\in\GL_d(k)$. We prove that if the characteristic polynomial is constant on $xH$, then the identity component of the Zariski closure of $H$ is solvable. In particular, $H$ is virtually solvable. Over $\C$, we give a topological proof using a null-homotopy of characteristic-polynomial fibers and a nonzero degree-three period on nontrivial representations of $\SU(2)$. An \'etale-cohomological proof establishes the result in arbitrary characteristic without a Levi decomposition. This proves Conjecture~9.3 of~\cite{Gelander} and implies Conjectures~9.5--9.7. The result also applies when the coset has only finitely many characteristic polynomials and gives a unipotence criterion in characteristic zero. \rev{As an application, we prove that every faithful finite-dimensional representation of $F_n$, $n\geq3$, over any field has infinitely many characteristic polynomials on primitive elements. This gives an alternative straightforward proof for the celebrated Formanek--Procesi nonlinearity theorem for $\Aut(F_n)$.} For complex representations and \rev{$n\geq5$}, the results of~\cite{PrimitiveSpectra} strengthen the classical proper-factor conclusion from nilpotent-by-abelian-by-finite to virtually unipotent, hence virtually nilpotent. Further consequences for primitive representations, automorphism actions, and invariant measures on compact representation spaces are also recorded from~\cite{PrimitiveSpectra}. Explicit examples describe the sharpness of the conclusions.
\end{abstract}

\section{Introduction and main theorem}

The main theorem is stated over an arbitrary field $k$. For a subgroup of $\GL_d(k)$, its Zariski closure is taken inside $\GL_d(\overline{k})$, where $\overline{k}$ is an algebraic closure of $k$, and is given its reduced structure. Subgroups need not be closed or finitely generated. We specify the field separately in sections using complex topology or explicit examples. For a matrix $A$, write
\[
\chi_A(z)=\det(zI-A).
\]
Unless otherwise specified, conjecture numbers refer to the April 2024 version of the survey~\cite{Gelander}.

Conjecture~9.3 of~\cite{Gelander} asks whether a subgroup $H\leq\GL_d(\C)$ must be virtually solvable whenever one of its cosets is contained in a single conjugacy class. We prove the following stronger statement.

\begin{theorem}[Constant characteristic polynomial]\label{thm:main}
Let $k$ be any field, let $H\leq\GL_d(k)$, and let $x\in\GL_d(k)$. Suppose that
\begin{equation}\label{eq:constant-spectrum}
\chi_{xh}=\chi_x\qquad\text{for every }h\in H.
\end{equation}
Then the identity component of $K=\Zcl{H}$ is solvable. In particular, $H$ is virtually solvable.
\end{theorem}

\begin{corollary}[Conjecture 9.3]\label{cor:93}
Let $k$ be any field, let $H\leq\GL_d(k)$, and let $x\in\GL_d(k)$. If $Hx$ is contained in the conjugacy class of $x$ in $\GL_d(\overline{k})$, then $H$ is virtually solvable. The same conclusion holds if $Hx$ is contained in the Zariski closure of that conjugacy class.
\end{corollary}

\begin{proof}
The characteristic polynomial is constant on a conjugacy class and on its Zariski closure. Moreover, $hx$ and $xh$ are conjugate. Apply Theorem~\ref{thm:main}.
\end{proof}

Two consequences will be used in the applications.

\begin{corollary}[Finitely many characteristic polynomials on a coset]\label{cor:finite-spectrum-coset}
Let $H\leq\GL_d(k)$ and $x\in\GL_d(k)$. If the set $\{\chi_{xh}:h\in H\}$ is finite, then $\bigl(\Zcl{H}\bigr)^\circ$ is solvable and $H$ is virtually solvable.
\end{corollary}

\begin{proof}
Let $K=\Zcl{H}$. The finite union of the relevant characteristic-polynomial fibers is Zariski closed, so it contains $xK$. The regular characteristic-polynomial map on the connected variety $xK^\circ$ has finite image and is therefore constant. At the identity of $K^\circ$ its value is $\chi_x$. Apply Theorem~\ref{thm:main} to the subgroup $K^\circ$.
\end{proof}

\begin{corollary}[Unipotent cosets in characteristic zero]\label{cor:unipotent-coset}
Assume that $\operatorname{char}(k)=0$. If $H$ is generated by unipotent matrices and $xH$ consists entirely of unipotent matrices, then $H$ is conjugate over $\overline{k}$ to a subgroup of the upper unitriangular group.
\end{corollary}

\begin{proof}
Put $K=\Zcl{H}$. For every unipotent generator $u$, the Zariski closure of $\langle u\rangle$ contains the one-parameter group $\{\exp(t\log u):t\in\overline{k}\}$: the exponential and logarithm are finite polynomial expressions, and a polynomial vanishing at all integral parameters vanishes identically. Thus every generator belongs to $K^\circ$, so $K$ is connected. Theorem~\ref{thm:main} makes $K$ solvable, and the Lie--Kolchin theorem simultaneously upper triangularizes it. Every unipotent generator has all diagonal entries equal to $1$, and hence so does every element of $H$.
\end{proof}

\begin{remark}[Conjectures 9.5--9.7]\label{rem:95-97}
Theorem~\ref{thm:main} also implies Conjectures~9.5--9.7 of~\cite{Gelander}. For $g\in\GL_d(\C)$, let
\[
\mathcal C_g=\{aga^{-1}g^{-1}:a\in\GL_d(\C)\}
\]
be the set of commutators with fixed second argument $g$. If $H\leq\GL_d(\C)$ is contained in $\Zcl{\mathcal C_g}$, then $Hg$ is contained in the Zariski closure of the conjugacy class of $g$, so Corollary~\ref{cor:93} makes $H$ virtually solvable. Thus no faithful image of $F_2$ is contained in $\Zcl{\mathcal C_g}$. Nor can this set contain a nontrivial homomorphic image of $\SL_2(\C)$: this group is perfect and has no nontrivial finite quotient, so any virtually solvable image is trivial. These conclusions hold also with $\SL_d(\C)$ in place of $\GL_d(\C)$ and without taking the closure, giving all three conjectures.
\end{remark}

The first proof, over $\C$, is topological. For any monic polynomial $p$ of degree $d$ with $p(0)\neq0$, the inclusion
\[
\mathcal F_p=\{A\in\GL_d(\C):\chi_A=p\}
\lhook\joinrel\longrightarrow\GL_d(\C)
\]
is null-homotopic. On the other hand, a nontrivial representation of $\SU(2)$, or any left translate of it, has a nonzero period against the closed three-form $\tr((A^{-1}dA)^3)$. A nonsolvable connected complex algebraic group contains the image of a nontrivial algebraic homomorphism from $\SL_2(\C)$. These observations prove the theorem without assumptions on the number or multiplicities of irreducible summands.

In positive characteristic, the same matrix interpolation yields vanishing in rational $\ell$-adic cohomology. A degree-three class detects every nontrivial algebraic representation of $\SL_2$, including Frobenius twists. Passing to the image on the semisimplification of the defining module replaces the Levi decomposition and proves the theorem in all characteristics.

\begin{revision}
The coset theorem also gives a strengthening of the nonlinearity theorem of Formanek and Procesi~\cite{FormanekProcesi}: every faithful finite-dimensional representation of $F_n$, $n\geq3$, over any field has infinite primitive characteristic-polynomial spectrum (Theorem~\ref{thm:infinite-primitive-spectrum}). The nonlinearity of $\Aut(F_n)$ follows because its primitive inner automorphisms are conjugate, so their images in any linear representation have a single characteristic polynomial.
\end{revision}

For \rev{$n\geq5$}, the application to $\Aut(F_n)$ admits a stronger conclusion over $\C$. The classical Formanek--Procesi theorem~\cite{FormanekProcesi} gives nilpotent-by-abelian-by-finite images of the inner automorphisms arising from proper free factors. Combining the coset theorem with the characteristic-representation results of~\cite{PrimitiveSpectra}, we obtain virtual unipotence, and hence virtual nilpotence, for these images in every dimension; see Corollary~\ref{cor:aut-proper-unipotent}. The finite-index qualification is necessary, as shown in Remark~\ref{rem:finite-image-obstruction}.

\begin{revision}
For $n\geq4$, the companion paper also proves that the connected solvable radical of the Zariski closure of the full inner image is unipotent. Consequently, virtual solvability, virtual nilpotence and virtual unipotence of the full inner image are equivalent. The remaining global conjecture is equivalent to finiteness of the inner image in every irreducible complex representation of $\Aut(F_n)$; see Corollary~\ref{cor:inner-radical-applications}.
\end{revision}

Section~\ref{sec:topology} gives the topological proof over $\C$, and Section~\ref{sec:positive-characteristic} proves the theorem in arbitrary characteristic. Section~\ref{sec:primitive-applications} records applications to primitive representations, characteristic representations, and invariant measures for compact groups; proofs of those results are given in the companion paper~\cite{PrimitiveSpectra}. \rev{Section~\ref{sec:aut-nonlinearity} proves infinitude of the primitive spectrum for faithful representations, deduces the Formanek--Procesi theorem, and gives the proper-factor virtual-unipotence strengthening for $n\geq5$ over $\C$.} Appendix~\ref{sec:examples} contains four examples illustrating the scope and sharpness of the theorem.

\section{The topological obstruction}\label{sec:topology}

Throughout this section the ground field is $\C$.

\begin{lemma}[Null-homotopy of a characteristic-polynomial fiber]\label{lem:fiber}
For every monic $p\in\C[z]$ of degree $d$ with $p(0)\neq0$, the inclusion $\mathcal F_p\hookrightarrow\GL_d(\C)$ is null-homotopic. More precisely, it is homotopic to the constant map with value $I$ by an explicit homotopy in $\GL_d(\C)$.
\end{lemma}

\begin{proof}
Let $\lambda_1,\ldots,\lambda_d\in\C^*$ be the roots of $p$, counted with multiplicity. For $A\in\mathcal F_p$ and $z\in\C$,
\begin{equation}\label{eq:qpoly}
\det((1-z)I+zA)
=q(z):=\prod_{j=1}^d(1-z+z\lambda_j).
\end{equation}
The right-hand side is independent of $A$. Since $q(0)=1$ and $q(1)=\prod_j\lambda_j\neq0$, choose a smooth path
\[
\gamma:[0,1]\longrightarrow\C\setminus q^{-1}(0),
\qquad\gamma(0)=0,\quad\gamma(1)=1.
\]
Then
\begin{equation}\label{eq:fiber-homotopy}
F(s,A)=(1-\gamma(s))I+\gamma(s)A
\end{equation}
is invertible for all $s$ and $A$, by \eqref{eq:qpoly}. It satisfies $F(0,A)=I$ and $F(1,A)=A$.
\end{proof}

\begin{remark}
The homotopy takes place in the ambient group $\GL_d(\C)$; it is not asserted to stay inside $\mathcal F_p$. In particular, the lemma does not assert that the fiber itself is contractible. Its possible singularities cause no difficulty, because \eqref{eq:fiber-homotopy} can be composed with any smooth map into the fiber.
\end{remark}

\begin{lemma}[A nonzero three-dimensional period]\label{lem:period}
Let $r:\SU(2)\to\GL_d(\C)$ be a nontrivial smooth homomorphism, and let $x\in\GL_d(\C)$. The map
\[
f_x:\SU(2)\longrightarrow\GL_d(\C),\qquad u\longmapsto x r(u),
\]
is not null-homotopic.
\end{lemma}

\begin{proof}
On $\GL_d(\C)$, let $\theta=A^{-1}dA$ be the Maurer--Cartan form and set
\[
\omega=\tr(\theta\wedge\theta\wedge\theta).
\]
Products of matrix-valued forms include both matrix multiplication and exterior multiplication. The Maurer--Cartan equation gives
\[
d\omega=-\tr(\theta^4)=0.
\]
The last equality follows from graded cyclicity of trace: moving the first factor of $\theta$ past the other three changes the sign, so $\tr(\theta^4)=-\tr(\theta^4)$. The form $\omega$ is left invariant, hence
\begin{equation}\label{eq:translation-period}
f_x^*\omega=r^*\omega.
\end{equation}

By averaging a Hermitian inner product, we may conjugate $r$ to a unitary representation. Conjugation preserves $\omega$. Choose a basis $U,V,W$ of $\sulie(2)$ such that $[V,W]=U$. At the identity,
\begin{align*}
(r^*\omega)_e(U,V,W)
&=3\tr\bigl(dr(U)[dr(V),dr(W)]\bigr)\\
&=3\tr\bigl(dr(U)^2\bigr).
\end{align*}
The differential $dr$ is nonzero, since $\SU(2)$ is connected and $r$ is nontrivial. Simplicity of $\sulie(2)$ implies that $dr$ is injective. Thus $dr(U)$ is a nonzero skew-Hermitian matrix, and
\[
\tr(dr(U)^2)<0.
\]
It follows that $r^*\omega$ is a nonzero left-invariant top-degree form on the compact connected three-manifold $\SU(2)$. Consequently,
\begin{equation}\label{eq:nonzero-period}
\int_{\SU(2)}r^*\omega\neq0.
\end{equation}
The integral of a closed form over a null-homotopic map from a closed oriented manifold is zero, by Stokes' theorem. Equations~\eqref{eq:translation-period} and~\eqref{eq:nonzero-period} prove the claim.
\end{proof}

\begin{proposition}[No semisimple coset in a characteristic-polynomial fiber]\label{prop:sl2-coset}
Let $r:\SL_2(\C)\to\GL_d(\C)$ be a nontrivial algebraic representation. For every $x\in\GL_d(\C)$, the characteristic polynomial is nonconstant on $x r(\SL_2(\C))$.
\end{proposition}

\begin{proof}
The restriction of $r$ to $\SU(2)$ is nontrivial: its differential complexifies to the differential of $r$. If the coset had constant characteristic polynomial, the map $u\mapsto xr(u)$ would take values in a single fiber $\mathcal F_p$. Lemma~\ref{lem:fiber} would make it null-homotopic, contrary to Lemma~\ref{lem:period}.
\end{proof}

\begin{proof}[Proof of Theorem~\ref{thm:main} over $\C$]
The coefficients of $\chi_{xh}$ are regular functions of $h$. Hence \eqref{eq:constant-spectrum} extends from $H$ to $K=\Zcl{H}$.

Suppose that $K^\circ$ is not solvable. By Levi decomposition in characteristic zero, a Levi subgroup of $K^\circ$ has a nontrivial connected semisimple derived subgroup. The root subgroup construction, applied to a simply connected cover of this subgroup, gives a nontrivial algebraic homomorphism
\[
r:\SL_2(\C)\longrightarrow K^\circ\lhook\joinrel\longrightarrow\GL_d(\C).
\]
These standard structure facts may be found in \cite{Borel}. The homomorphism may have a finite central kernel; injectivity is not needed. Since $xr(\SL_2(\C))\subseteq xK$, its characteristic polynomial is constant, contradicting Proposition~\ref{prop:sl2-coset}.

Thus $K^\circ$ is solvable. The subgroup $H\cap K^\circ$ is solvable and has finite index in $H$, completing the proof.
\end{proof}

\section{The theorem in arbitrary characteristic}\label{sec:positive-characteristic}

We now prove Theorem~\ref{thm:main} over every field, including fields of positive characteristic. Throughout this section, $k$ is algebraically closed and $\ell$ is a prime different from $\operatorname{char}(k)$. Write
\[
H^i(X)=H^i_{\mathrm{\acute et}}(X,\mathbb Q_\ell).
\]
All algebraic groups, unless a group scheme is explicitly mentioned, are smooth affine algebraic groups over $k$. We suppress Tate twists, since no Galois action is being considered. The cohomological input consists of the K\"unneth formula, homotopy invariance, the localization sequence with purity, and finite \'etale transfer; see \cite[\S\S16, 19, 22]{MilneEtale}. We give the degree-three calculation explicitly.

\subsection{Connected families and the degree-three class}

\begin{lemma}[Connected families]\label{lem:connected-family}
Let $S$ be a connected variety over $k$, and let $F:S\times X\to Y$ be a morphism of varieties. For any $s_0,s_1\in S(k)$, the maps $F_{s_0}$ and $F_{s_1}$ induce the same pullback on $H^i$ for every $i$. In particular, left and right translations on a connected algebraic group act trivially on its cohomology.
\end{lemma}

\begin{proof}
Under the K\"unneth decomposition of $H^i(S\times X)$, restriction to $\{s\}\times X$ kills the summands of positive degree on $S$ and identifies $H^0(S)\otimes H^i(X)$ with $H^i(X)$. Since $H^0(S)=\mathbb Q_\ell$, this is independent of $s$. Apply this to $F^*$, and then to group multiplication.
\end{proof}

\begin{lemma}[Vanishing on a characteristic-polynomial fiber]\label{lem:etale-fiber}
For a monic polynomial $P\in k[z]$ with $P(0)\neq0$, let
\[
\mathcal F_P=\{A\in\GL_d(k):\chi_A=P\}.
\]
The inclusion $i_P:\mathcal F_P\hookrightarrow\GL_d$ induces the zero map on $H^i$ for every $i>0$.
\end{lemma}

\begin{proof}
If $\lambda_1,\ldots,\lambda_d$ are the roots of $P$, counted with multiplicity, set
\[
q(t)=\prod_{j=1}^d(1-t+t\lambda_j),\qquad
S=\mathbb A^1_k\setminus q^{-1}(0).
\]
The curve $S$ is connected and contains $0$ and $1$. As in \eqref{eq:qpoly},
\[
F:S\times\mathcal F_P\longrightarrow\GL_d,
\qquad F(t,A)=(1-t)I+tA
\]
is a morphism. Its restrictions at $0$ and $1$ are the constant map with value $I$ and $i_P$, respectively. Lemma~\ref{lem:connected-family} proves the assertion. The fiber need not be smooth.
\end{proof}

\begin{lemma}[The stable degree-three class]\label{lem:etale-generator}
For $d\geq2$, one has $H^2(\GL_d)=0$ and $\dim H^3(\GL_d)=1$. Standard block inclusions $\GL_d\hookrightarrow\GL_{d+1}$ and the inclusion $\SL_2\hookrightarrow\GL_2$ induce isomorphisms on $H^3$.

Fix a nonzero class $\eta\in H^3(\SL_2)$, and let $\eta_d\in H^3(\GL_d)$ be the compatible classes restricting to $\eta$. With $\eta_1=0$, these classes satisfy
\begin{equation}\label{eq:etale-primitive}
m^*\eta_d=\operatorname{pr}_1^*\eta_d+
\operatorname{pr}_2^*\eta_d,
\end{equation}
where $m$ is multiplication on $\GL_d$.
\end{lemma}

\begin{proof}
Localization and purity give
\[
H^i(\mathbb A^N_k\setminus\{0\})=
\begin{cases}
\mathbb Q_\ell,&i=0,\ 2N-1,\\
0,&\text{otherwise},
\end{cases}
\]
with the Tate twist omitted. The last-column map
\[
\GL_N\longrightarrow\mathbb A^N_k\setminus\{0\}
\]
is Zariski locally trivial, with fiber the stabilizer of the last basis vector, isomorphic as a variety to $\GL_{N-1}\times\mathbb A^{N-1}$. Its structure group is this connected stabilizer. Lemma~\ref{lem:connected-family} shows that its action on the cohomology of the fiber is trivial, so the higher direct image sheaves are constant.

For $N=2$, the fiber has cohomology only in degrees $0$ and $1$. The Leray spectral sequence therefore gives $H^2(\GL_2)=0$ and identifies $H^3(\GL_2)$ with the pullback of $H^3(\mathbb A^2\setminus\{0\})$. The corresponding map $\SL_2\to\mathbb A^2\setminus\{0\}$ is a locally trivial $\mathbb A^1$-bundle and induces an isomorphism on cohomology. This proves the assertion about $\SL_2\hookrightarrow\GL_2$.

For $N\geq3$, the base has zero cohomology in degrees $1$ through $4$. The same spectral sequence identifies $H^i(\GL_N)$ with the cohomology of its fiber for $i=2,3$. Homotopy invariance then gives the asserted stability under the standard block inclusion. Finally, since $H^2(\GL_d)=0$, the K\"unneth decomposition in degree three on $\GL_d\times\GL_d$ has only the two summands of bidegrees $(3,0)$ and $(0,3)$. Restriction of $m$ to either factor is the identity, proving \eqref{eq:etale-primitive}.
\end{proof}

\subsection{The index of an \texorpdfstring{$\SL_2$}{SL2}-representation}

\begin{proposition}[A nonzero degree-three pullback in every characteristic]\label{prop:etale-index}
Let $r:\SL_2\to\GL_d$ be an algebraic representation, and let $n_1,\ldots,n_d\in\Z$ be its weights on the torus $\diag(t,t^{-1})$. Then
\begin{equation}\label{eq:etale-index}
r^*\eta_d=D(r)\eta,
\qquad D(r)=\frac12\sum_{i=1}^d n_i^2.
\end{equation}
If $r$ is nontrivial, $D(r)$ is a positive integer and hence is nonzero in $\mathbb Q_\ell$.
\end{proposition}

\begin{proof}
Define $D(r)\in\mathbb Q_\ell$ by the first equality in \eqref{eq:etale-index}. We first establish its behavior on exact sequences and tensor products. For a short exact sequence of representations, choose a basis adapted to the invariant subspace and write
\[
r(g)=\begin{pmatrix}r_1(g)&B(g)\\0&r_2(g)\end{pmatrix}.
\]
The matrices
\[
r_t(g)=\begin{pmatrix}r_1(g)&tB(g)\\0&r_2(g)\end{pmatrix}
\qquad(t\in\mathbb A^1)
\]
define a morphism $\mathbb A^1\times\SL_2\to\GL_d$, specializing to $r$ and $r_1\oplus r_2$. Lemma~\ref{lem:connected-family}, block stability, and \eqref{eq:etale-primitive} imply
\begin{equation}\label{eq:index-additive}
D(r)=D(r_1)+D(r_2).
\end{equation}
Indeed, a block sum is the pointwise product of the two representations placed in complementary blocks. These block inclusions have the same pullback as the standard ones, since conjugation acts trivially on cohomology.

Likewise, writing a tensor product as the pointwise product of $r_V\otimes I_W$ and $I_V\otimes r_W$ gives
\begin{equation}\label{eq:index-tensor}
D(V\otimes W)=\dim(W)D(V)+\dim(V)D(W).
\end{equation}
Here each factor is conjugate to a direct sum of copies of the indicated representation. For the standard two-dimensional module $E$, the normalization gives $D(E)=1$.

Let $R$ be the Grothendieck ring of finite-dimensional rational $\SL_2$-modules, with relations from all short exact sequences. The formal character identifies
\begin{equation}\label{eq:sl2-character-ring}
R\simeq\Z[z,z^{-1}]^{z\leftrightarrow z^{-1}}
=\Z[z+z^{-1}],
\qquad [E]\longmapsto z+z^{-1}.
\end{equation}
This holds in arbitrary characteristic; see \cite[Theorem~22.38]{MilneGroups}. To recall why it is enough here, highest weights make the characters of simple modules linearly independent, so the character map is injective. Weyl symmetry puts its image in the displayed invariant ring, and the character of $E$ generates that ring, giving surjectivity.

Equations~\eqref{eq:index-additive} and~\eqref{eq:index-tensor} say that $D$ is a derivation of $R$ at the dimension homomorphism. On formal characters, define
\[
S(f)=\frac12\left.\left(z\frac{d}{dz}\right)^2f(z)\right|_{z=1}.
\]
These are derivatives of Laurent polynomials over $\mathbb Q_\ell$, not derivatives of functions over $k$. Symmetry under $z\leftrightarrow z^{-1}$ gives $f'(1)=0$. Thus $S$ obeys the same product rule as $D$, with dimension given by evaluation at $1$, and $S(z+z^{-1})=1$. The polynomial-ring description \eqref{eq:sl2-character-ring} gives $D=S$, proving the formula.

The weights occur in opposite pairs, so $\frac12\sum_i n_i^2$ is a nonnegative integer. If all weights vanished, the torus would lie in the kernel of $r$. Choose $a\in k^*$ with $a^2\neq1$. Commuting $\diag(a,a^{-1})$ with upper and lower unipotent elements shows that both root subgroups would also lie in the kernel. They generate $\SL_2$, so $r$ would be trivial. This proves strict positivity for a nontrivial representation.
\end{proof}

\begin{remark}[Frobenius twists]\label{rem:frobenius-index}
In characteristic $p>0$, the $e$th Frobenius twist of the standard representation has weights $p^e,-p^e$, and hence index $p^{2e}\neq0$ in $\mathbb Q_\ell$. Its differential is zero when $e>0$. Thus the obstruction detects representations that an argument using only differentials would miss. The weights in \eqref{eq:etale-index} are integer characters of a torus; they are not reduced modulo $p$.
\end{remark}

\subsection{Unipotent kernels and the proof of the theorem}

In positive characteristic a connected affine group need not have a Levi subgroup. The following observation allows us to pass instead to the image on the semisimplification of its defining module.

\begin{lemma}[Pullback through a unipotent kernel]\label{lem:unipotent-kernel}
Let $q:G\to L$ be a surjective homomorphism of connected affine algebraic groups. If its scheme-theoretic kernel is unipotent, then
\[
q^*:H^i(L)\longrightarrow H^i(G)
\]
is injective for every $i$.
\end{lemma}

\begin{proof}
Let $N=\ker(q)$ and let $U=(N_{\mathrm{red}})^\circ$. Over the algebraically closed field $k$, this is a smooth connected normal unipotent subgroup of $G$. Such a group is split, is isomorphic as a variety to an affine space, and its torsors are Zariski locally trivial. Consequently, the quotient map $G\to G/U$ induces an isomorphism on cohomology.

The remaining map $G/U\to L$ is an isogeny, with finite kernel $N/U$. Quotienting first by the connected component of this finite group scheme factors the isogeny as a finite universal homeomorphism followed by a finite \'etale covering. The first induces an isomorphism on \'etale cohomology. Pullback for the second is injective, because its composite with transfer is multiplication by the degree of the covering, an invertible scalar in $\mathbb Q_\ell$. This also explains why inseparability of $q$ causes no difficulty. The same isogeny argument is given with finite coefficients in \cite[Proposition~3.3]{ChalupnikKowalski}.
\end{proof}

\begin{proof}[Proof of Theorem~\ref{thm:main} in arbitrary characteristic]
Extend the original field to its algebraic closure $k$, and let $K$ be the Zariski closure of $H$ in $\GL_d(k)$. The coefficients of the characteristic polynomial are regular functions, so the hypothesis extends from $H$ to $K$. Put $G=K^\circ$ and let $j:G\hookrightarrow\GL_d$ be the inclusion. The case $d=1$ is immediate, so assume $d\geq2$.

The map $g\mapsto xj(g)$ factors through $\mathcal F_{\chi_x}$. Lemma~\ref{lem:etale-fiber} and translation invariance give
\begin{equation}\label{eq:vanishing-defining-class}
j^*\eta_d=0.
\end{equation}
Choose a composition series of the $G$-module $k^d$ and a basis adapted to it. In this basis $j$ is block upper triangular. Let $j_{\mathrm{ss}}$ be the block diagonal representation on the direct sum of the simple quotients. If there are $s$ blocks, conjugation by the block scalar matrix with exponents $s-1,s-2,\ldots,0$ multiplies block $(a,b)$, for $a<b$, by $t^{b-a}$. This extends to a morphism $\mathbb A^1\times G\to\GL_d$, specializing to $j$ at $1$ and $j_{\mathrm{ss}}$ at $0$. Hence
\begin{equation}\label{eq:semisimplified-class}
j_{\mathrm{ss}}^*\eta_d=j^*\eta_d=0.
\end{equation}

Let $L=j_{\mathrm{ss}}(G)$ be the closed image, with its reduced structure, and write $j_{\mathrm{ss}}=i\circ q$, where $q:G\to L$ is surjective and $i:L\hookrightarrow\GL_d$. The kernel of $q$ is contained in the block upper unitriangular group and is therefore unipotent. Moreover, $L$ is connected and reductive: its defining module is faithful and completely reducible, and the unipotent radical acts trivially on every simple summand. Indeed, the fixed space of a normal unipotent subgroup in a simple module is nonzero and invariant, hence is the whole module.

Suppose that $G$ is not solvable. Since $\ker(q)$ is unipotent, $L$ is not solvable either. A simply connected cover of the derived group of this connected reductive group supplies a nontrivial root homomorphism
\[
\alpha:\SL_2\longrightarrow L;
\]
see \cite{Borel,MilneGroups}. The representation $i\circ\alpha$ is nontrivial, so Proposition~\ref{prop:etale-index} gives
\[
\alpha^*i^*\eta_d\neq0.
\]
In particular $i^*\eta_d\neq0$. By Lemma~\ref{lem:unipotent-kernel}, $q^*$ is injective, and therefore
\[
j_{\mathrm{ss}}^*\eta_d=q^*i^*\eta_d\neq0,
\]
contradicting \eqref{eq:semisimplified-class}. Thus $K^\circ$ is solvable. As before, $H\cap K^\circ$ is a solvable subgroup of finite index in $H$.
\end{proof}

\begin{remark}
The argument imposes no restriction on the characteristic or its size relative to $d$. It uses rational $\ell$-adic coefficients with $\ell$ different from the characteristic, so a nonzero integer index cannot disappear. No Levi decomposition or complete reducibility of the original representation is assumed.
\end{remark}

\section{Applications to variants of Wiegold's conjecture}\label{sec:primitive-applications}

We record, without proofs, results on the linear variants of Wiegold's conjecture discussed in \cite[Sections~2 and~4--8]{Gelander}, following our companion paper~\cite{PrimitiveSpectra}. Corollaries~\ref{cor:finite-spectrum-coset} and~\ref{cor:unipotent-coset} supply the input on proper free factors. Fixed-space and Nielsen arguments yield the quantitative unipotence criteria, and property~(T) gives virtual unipotence on proper factors for finite automorphism orbits when \rev{$n\geq5$}. Elementary centralizer arguments, independent of the coset theorem and property~(T), give Jordan restrictions and primitive quasi-unipotence already for $n\geq3$. \rev{For $n\geq4$, the connected solvable radical of the inner-image closure is unipotent; the remaining obstruction lies in its semisimple quotient. We record this reduction and the constraints imposed on that quotient by Nielsen subgroups.} Subsection~\ref{sec:compact-measure-applications} gives the quantitative classification of invariant measures for compact groups, \rev{together with dense-orbit conclusions}. \rev{These results and the rank-three spectral criterion in Subsection~\ref{sec:compact-spectral-applications} are proved separately by compact redundancy methods.}

Throughout this section, \rev{$n\geq3$ unless otherwise stated}, the ground field is $\C$, and $P_n\subset F_n$ is the set of primitive elements, namely those belonging to a free basis. For a representation $\rho:F_n\to\GL_d(\C)$, put
\[
\mathcal S(\rho)=\{\chi_{\rho(p)}:p\in P_n\},
\qquad K=\Zcl{\rho(F_n)}.
\]
\rev{We call $\mathcal S(\rho)$ the \emph{primitive characteristic-polynomial spectrum}, or simply the \emph{primitive spectrum}, of $\rho$.} We call $\rho$ \emph{primitive-unipotent} if every $\rho(p)$, $p\in P_n$, is unipotent, and \emph{unipotent} if its whole image is conjugate into the upper unitriangular group. We call $\rho$ \emph{Zariski redundant} if some proper free factor $A<F_n$ satisfies $\Zcl{\rho(A)}=K$.

\subsection{Finite primitive spectra and proper free factors}

Zelmanov's conjecture, as stated in \cite[Section~4]{Gelander}, asks whether $\#\mathcal S(\rho)=1$ forces $\rho(F_n)$ to be virtually solvable. The primitive-torsion variant asks for the same conclusion when every primitive image has finite order; it is Conjecture~6.2 of \cite{Gelander}. The following theorem gives a common proper-factor restriction.

\Needspace{9\baselineskip}
\begin{theorem}[Proper-free-factor restrictions]\label{thm:primitive-factors}
Let $\rho:F_n\to\GL_d(\C)$.
\begin{enumerate}
\item If $\mathcal S(\rho)$ is finite, then $\rho(A)$ is virtually solvable for every proper free factor $A<F_n$. This applies, in particular, if all primitive images have one characteristic polynomial or if all primitive images have finite order.
\item If $\rho$ is primitive-unipotent, then $\rho(A)$ is unipotent for every proper free factor $A<F_n$.
\item Suppose that $\mathcal S(\rho)$ is finite and that $A<F_n$ is a proper free factor of rank at least two with connected Zariski closure $\Zcl{\rho(A)}$. Then $\rho(A)$ is unipotent. If, in addition, every primitive image has finite order, then $\rho(A)=1$.
\end{enumerate}
\end{theorem}

Compare the restrictions obtained by Formanek and Procesi~\cite{FormanekProcesi} for linear representations of their HNN extensions. The hypotheses here concern only the primitive spectra of a representation of $F_n$; no extension to $\Aut(F_n)$ is assumed. The restrictions apply to all proper free factors, including those obtained after arbitrary Nielsen changes of basis. They give the following conclusions when a proper factor has the full Zariski closure.

\Needspace{9\baselineskip}
\begin{corollary}[The Zariski-redundant case]\label{cor:primitive-redundant}
Let $\rho:F_n\to\GL_d(\C)$ be Zariski redundant.
\begin{enumerate}
\item If $\mathcal S(\rho)$ is finite, then $K^\circ$ is unipotent. In particular, $\rho(F_n)$ is virtually unipotent.
\item If $\mathcal S(\rho)$ consists of one polynomial, then $\rho(F_n)$ is unipotent.
\item If every primitive image has finite order, then $\rho(F_n)$ is finite.
\end{enumerate}
\end{corollary}

Here \emph{virtually unipotent} means that a subgroup of finite index is conjugate into the upper unitriangular group. In particular, a Zariski-redundant representation with nonsolvable $K^\circ$ has infinitely many primitive characteristic polynomials.

\begin{revision}
A further consequence of Theorem~\ref{thm:primitive-factors}, recorded in \cite[Corollary~7.2]{PrimitiveSpectra}, is that a representation with one primitive characteristic polynomial is primitive-unipotent as soon as some proper factor of rank at least two has connected Zariski closure. The unipotence criteria below then apply to the whole image.
\end{revision}

\subsection{Quantitative unipotence results}

The Platonov--Potapchik conjecture asserts that every primitive-unipotent representation is unipotent. Write $\PP(n,d)$ and $U(n,d)$ for the varieties of primitive-unipotent and unipotent representations, respectively. The next estimate combines the proper-factor conclusion of Theorem~\ref{thm:primitive-factors} with fixed-space inequalities and simultaneous minima on Nielsen orbit closures.

\Needspace{11\baselineskip}
\begin{theorem}[Irreducible-constituent bound]\label{thm:primitive-constituents}
Let $\sigma:F_n\to\GL_m(\C)$ be a nontrivial irreducible primitive-unipotent representation, and set
\[
b=\min_{p\in P_n}\dim\ker(\sigma(p)-I).
\]
Then
\[
b\geq n,
\qquad
\rev{m\geq b+2\left\lfloor\frac{b}{n-1}\right\rfloor+2\geq n+4.}
\]
In particular, every nontrivial irreducible constituent of a primitive-unipotent representation has dimension at least $n+4$.
\rev{For $n=3$, the sharper estimate is $m\geq2b+1$ when $b$ is odd and $m\geq2b+2$ when $b$ is even; see \cite[Proposition~4.4]{PrimitiveSpectra}.}
\end{theorem}

\Needspace{11\baselineskip}
\begin{corollary}[Dimension, Jordan-block and rank criteria]\label{cor:primitive-quantitative}
Let $\rho:F_n\to\GL_d(\C)$ be primitive-unipotent. Its whole image is unipotent if any one of the following conditions holds:
\begin{enumerate}
\item $d\leq n+3$;
\item some primitive image has at most $n-1$ Jordan blocks, with no restriction on their sizes;
\item $\rank(\rho(p)-I)\leq3$ for every $p\in P_n$, with no restriction on $d$.
\end{enumerate}
Moreover, for $d\leq n+3$ one has $\PP(n,d)=U(n,d)$, and this variety is irreducible. Thus Conjectures~7.1 and~7.3 of \cite{Gelander} hold in this range.
\end{corollary}

For $F_3$, the dimension criterion gives $d\leq6$, while one primitive image with at most two Jordan blocks suffices in every dimension. If all primitive images are conjugate, the rank criterion can be checked on a single primitive image. These statements concern representations of $F_n$ itself and do not require an extension to $\Aut(F_n)$.

\begin{revision}
The companion paper proves the following further statements in arbitrary dimension; see \cite[Theorems~1.1, 3.6 and~5.2 and Corollary~6.2]{PrimitiveSpectra}.

\Needspace{9\baselineskip}
\begin{theorem}[Reduction to one primitive conjugacy class]\label{thm:primitive-conjugacy-reduction}
Fix $n\geq2$ and $d\geq1$. If $\rho\in\PP(n,d)$ and $\rho(w)$ is non-unipotent for some $w\in F_n$, there is $\sigma\in\PP(n,d)$ for which $\sigma(w)$ is non-unipotent and all primitive images are conjugate. Thus the primitive-unipotent conjecture in each rank and dimension is equivalent to its restriction to representations with one primitive conjugacy class.
\end{theorem}

\Needspace{9\baselineskip}
\begin{theorem}[Further unipotence criteria]\label{thm:additional-unipotence}
Let $\rho:F_n\to\GL_d(\C)$ be primitive-unipotent.
\begin{enumerate}
\item For every $n\geq2$, if $\rank((\rho(p)-I)^2)\leq1$ for every primitive $p$, then $\rho$ is unipotent. Equivalently, each primitive image may have blocks of sizes one and two and at most one block of size three.
\item For $n=2$, it suffices that $\rank(\rho(p)-I)\leq2$ for one primitive $p$.
\item For $n\geq3$, if every rank-two free-factor image is nilpotent of class at most two, then $\rho(F_n)$ is unipotent and nilpotent of class at most two.
\end{enumerate}
\end{theorem}
\end{revision}

\subsection{Characteristic representations and primitive inner images}\label{sec:characteristic-applications}

A representation $\rho:F_n\to\GL_d(\C)$ is \emph{characteristic} if for every $\alpha\in\Aut(F_n)$ there exists $S_\alpha\in\GL_d(\C)$ such that
\[
\rho(\alpha(w))=S_\alpha\rho(w)S_\alpha^{-1}
\qquad(w\in F_n).
\]
More generally, $\rho$ has a \emph{finite automorphism orbit up to simultaneous conjugacy} if the representations $\rho\circ\alpha$, $\alpha\in\Aut(F_n)$, belong to finitely many simultaneous conjugacy classes. This concerns conjugacy of the whole representation, with no semisimplicity assumption. A matrix is \emph{quasi-unipotent} if some positive power is unipotent, equivalently if all its eigenvalues are roots of unity.

Write $\iota_w(v)=wvw^{-1}$, with automorphisms composed from right to left. For every primitive $p\in F_n$ and $n\geq3$, the companion paper proves the centralizer identity
\begin{equation}\label{eq:primitive-centralizer}
\iota_p\in
[C_{\Aut(F_n)}(\iota_p),C_{\Aut(F_n)}(\iota_p)].
\end{equation}
Here the brackets denote the abstract commutator subgroup, without taking a closure. This identity gives the following restrictions on the full Jordan form of a primitive inner image; see \rev{\cite[Theorem~1.3 and Section~8.2]{PrimitiveSpectra}}.

\Needspace{13\baselineskip}
\begin{theorem}[Jordan restrictions for primitive inner images]\label{thm:primitive-inner-jordan}
Let $n\geq3$, let $R:\Aut(F_n)\to\GL_d(\C)$, and let $p\in F_n$ be primitive. Put $A=R(\iota_p)$ and let $m_{\lambda,s}$ denote the number of size-$s$ Jordan blocks of $A$ with eigenvalue $\lambda$. Then
\[
A\in[C_{\GL_d(\C)}(A),C_{\GL_d(\C)}(A)],
\]
and the following hold:
\begin{enumerate}
\item $\lambda^{m_{\lambda,s}}=1$ whenever $m_{\lambda,s}>0$;
\item for each eigenvalue $\lambda$, the occurring block sizes are exactly $1,2,\ldots,r_\lambda$, where $r_\lambda$ is the largest such size.
\end{enumerate}
In particular, $A$ is quasi-unipotent. If $o_\lambda$ is the order of $\lambda$ and $E_\lambda$ is its generalized eigenspace, then
\begin{equation}\label{eq:primitive-jordan-bound}
\dim E_\lambda\geq o_\lambda\frac{r_\lambda(r_\lambda+1)}2.
\end{equation}
If $A$ is unipotent with largest block size $r$, then $d\geq r(r+1)/2$.
\end{theorem}

These restrictions use an actual representation of $\Aut(F_n)$. In particular, the unipotent types $(5,1,1)$, $(4,2,1)$ and $(3,3,1)$ cannot occur as primitive inner images. This does not exclude them for arbitrary primitive-unipotent representations of $F_3$, and does not settle the general dimension-seven problem.

The same centralizer argument, applied through the canonical action on the image algebra, yields the following conclusion for characteristic representations and finite automorphism orbits; see \rev{\cite[Proposition~8.7]{PrimitiveSpectra}}.

\Needspace{10\baselineskip}
\begin{theorem}[Primitive quasi-unipotence in every rank at least three]\label{thm:finite-orbit-quasi-unipotence}
Let $n\geq3$ and let $\rho:F_n\to\GL_d(\C)$ have finite $\Aut(F_n)$-orbit up to simultaneous conjugacy. There is an integer $q\geq1$, depending on $\rho$, such that
\[
\rho(p)^q\quad\text{is unipotent for every }p\in P_n.
\]
In particular, the common primitive conjugacy class of every characteristic representation is quasi-unipotent. These assertions hold in every dimension.
\end{theorem}

The exponent is uniform over the primitive elements of a given representation. No exponent independent of the representation is asserted. In particular, the theorem applies to $\rho(w)=R(\iota_w)$ for every complex linear representation $R$ of $\Aut(F_n)$. Its rank threshold does not use property~(T).

\rev{For virtual unipotence of whole proper-factor images, the argument in \cite[Section~8.4]{PrimitiveSpectra} uses property~(T) for $\Aut(F_k)$ at every $k\geq4$, including Nitsche\textquotesingle s rank-four theorem. It yields the following stronger statement about the connected solvable radical.}

\Needspace{6\baselineskip}
\begin{proposition}[Virtually solvable characteristic images]\label{prop:characteristic-solvable}
Let \rev{$k\geq4$} and let $\sigma:F_k\to\GL_d(\C)$ be characteristic.
\rev{For $L=\Zcl{\sigma(F_k)}$, the connected solvable and unipotent radicals satisfy $R(L^\circ)=R_u(L^\circ)$.}
If $\sigma(F_k)$ is virtually solvable, then it is virtually unipotent.
\end{proposition}

Theorem~\ref{thm:main} supplies virtual solvability on proper free factors. Combining this input with the preceding proposition gives the following result, including finite automorphism orbits.

\Needspace{10\baselineskip}
\begin{theorem}[Proper-factor virtual unipotence at fixed rank]\label{thm:finite-orbit-applications}
Let \rev{$n\geq5$} and $d\geq1$, and suppose that $\rho:F_n\to\GL_d(\C)$ has finite $\Aut(F_n)$-orbit up to simultaneous conjugacy. Then, for every proper free factor $A<F_n$,
\[
\bigl(\Zcl{\rho(A)}\bigr)^\circ
\quad\text{is unipotent}.
\]
In particular, $\rho(A)$ is virtually unipotent. These conclusions apply to every characteristic representation, with no restriction on $d$.
\end{theorem}

\rev{Thus $n=5$ suffices in every dimension for virtual unipotence on proper factors; see \cite[Theorem~1.2]{PrimitiveSpectra}.} Theorem~\ref{thm:finite-orbit-quasi-unipotence} gives primitive quasi-unipotence already for $n\geq3$. No finite-index bound independent of the representation is asserted.

For representations of $\Aut(F_n)$, this yields the strengthening of the Formanek--Procesi theorem stated in Corollary~\ref{cor:aut-proper-unipotent} below.

The companion paper identifies Conjectures~5.1, 8.3 and~8.5 of~\cite{Gelander} at every fixed rank $n\geq3$; see \rev{\cite[Proposition~8.7]{PrimitiveSpectra}}.

\Needspace{13\baselineskip}
\begin{proposition}[Equivalence of three conjectures]\label{prop:characteristic-equivalences}
For each fixed $n\geq3$, the following assertions are equivalent, with the dimension unrestricted:
\begin{enumerate}
\item For every complex linear representation $R:\Aut(F_n)\to\GL(W)$, the group $R(\Inn(F_n))$ is virtually solvable.
\item Every characteristic complex linear representation of $F_n$ has virtually solvable image.
\item Every characteristic complex linear representation of $F_n$ whose common primitive conjugacy class is quasi-unipotent has virtually solvable image.
\end{enumerate}
These are equivalent formulations of Conjectures~5.1, 8.3 and~8.5 of \cite{Gelander}, respectively. The extra hypothesis in the third assertion is automatic for every $n\geq3$ by Theorem~\ref{thm:finite-orbit-quasi-unipotence}.
\end{proposition}

\begin{revision}
The following further reduction holds already for $n\geq4$; it is \cite[Corollary~8.9]{PrimitiveSpectra}. Write $R$ and $R_u$ for the connected solvable and unipotent radicals.

\Needspace{11\baselineskip}
\begin{corollary}[The remaining inner-image obstruction]\label{cor:inner-radical-applications}
Let $n\geq4$, let $f:\Aut(F_n)\to\GL(V)$ be a complex representation, and put $K=\Zcl{f(\Inn(F_n))}$. Then
\[
R(K^\circ)=R_u(K^\circ).
\]
Consequently, virtual solvability, virtual nilpotence and virtual unipotence of $f(\Inn(F_n))$ are equivalent. If $f$ is completely reducible, then $K^\circ$ is semisimple and these conditions are equivalent to finiteness of the inner image.

For each fixed $n\geq4$, Conjecture~5.1 of~\cite{Gelander} is therefore equivalent to the assertion that every irreducible finite-dimensional complex representation of $\Aut(F_n)$ has finite inner image.
\end{corollary}
\end{revision}

\begin{remark}[Scope of the conclusions]\label{rem:characteristic-scope}
Proposition~\ref{prop:characteristic-equivalences} identifies conjectures; their global conclusion is not proved here. The Jordan restrictions, primitive quasi-unipotence, and equivalence of the three conjectures hold for $n\geq3$. Virtual unipotence on every proper free factor is established for \rev{$n\geq5$}, but does not establish virtual solvability of the full inner image. Conjugacy of individual primitive images alone does not imply that a representation is characteristic, so these results do not settle the unrestricted Platonov--Potapchik or Zelmanov conjectures.
\end{remark}

\subsection{Nielsen subgroups in the semisimple inner quotient}\label{sec:nielsen-inner-quotient}

The following result retains the full automorphism action after removing the connected solvable radical of the inner closure. It constrains a hypothetical counterexample to the Formanek--Zelmanov--Lubotzky conjecture; see \rev{\cite[Proposition~8.12]{PrimitiveSpectra}}.

\Needspace{14\baselineskip}
\begin{proposition}[An isomorphism forced by Nielsen relations]\label{prop:nielsen-inner-isomorphism}
Let $n\geq3$, let $f:\Aut(F_n)\to\GL(V)$ be a complex linear representation, and put
\[
K=\Zcl{f(\Inn(F_n))},\qquad S=K^\circ/R(K^\circ),
\]
where $R(K^\circ)$ is the connected solvable radical. Suppose $S\neq1$, and let
\[
\psi:\Aut(F_n)\longrightarrow\Aut_{\mathrm{alg}}(S)
\]
be the action induced by conjugation. For a free basis $x_1,\ldots,x_n$, put
\[
B_i=\langle x_j:j\neq i\rangle,\qquad
 a_i=\psi(\iota_{x_i}),\qquad H_i=\Zcl{\psi(\iota_{B_i})},
\]
where $\iota_{B_i}=\{\iota_w:w\in B_i\}$. Let $R_{i,w}$ send $x_i$ to $x_iw$ and fix $B_i$ pointwise, and set
\[
T_i=\Zcl{\{\psi(R_{i,w}):w\in B_i\}}.
\]
All these closures are taken in the linear algebraic group $\Aut_{\mathrm{alg}}(S)$. Then $T_i$ centralizes $H_i$, and
\begin{equation}\label{eq:nielsen-inner-isomorphism}
c_i:T_i\longrightarrow H_i,\qquad
c_i(t)=a_i^{-1}t a_i t^{-1}
\end{equation}
is an isomorphism of algebraic groups. Both identity components are solvable for $n\geq3$, and unipotent for \rev{$n\geq5$}. The analogous conclusions hold for the left Nielsen subgroups.
\end{proposition}

The solvability assertion uses Theorem~\ref{thm:main}; the unipotence assertion uses Theorem~\ref{thm:finite-orbit-applications}. \rev{For $n\geq4$, Corollary~\ref{cor:inner-radical-applications} identifies $S$ with $K^\circ/R_u(K^\circ)$, and virtual unipotence of the full inner image is equivalent to $S=1$.} These constraints do not rule out $S\neq1$: unipotent subgroups can generate a connected semisimple group. The remaining issue is compatibility of the Nielsen relations for all indices and all free bases. \rev{The global Formanek--Zelmanov--Lubotzky conjecture remains unresolved by these results.}

\subsection{Quantitative classification of measures for compact groups}\label{sec:compact-measure-applications}

Let $G$ be a \rev{compact Lie group} and fix a faithful unitary representation $G\leq\UU(m)$. The integer $m$ is the degree of this representation. Identify $\Hom(F_n,G)$ with $G^n$ using a free basis; $\Aut(F_n)$ acts by precomposition. Throughout this subsection, closures and density refer to the compact topology. Write
\[
\Epi(F_n,H)=\{\rho\in\Hom(F_n,H):\overline{\rho(F_n)}=H\}
\]
for any closed subgroup $H\leq G$.

The results below are the refinement of Cantat--Dupont--Martin-Baillon's measure theorem~\cite{CDB} established in \rev{\cite[Theorem~9.2 and Corollary~9.6]{PrimitiveSpectra}}. Their proof uses compression of finite component groups and two redundant generators, independently of Theorem~\ref{thm:main}.

Let $J(m)$ be the least integer such that every finite subgroup of $\GL_m(\C)$ has an abelian normal subgroup of index at most $J(m)$, and set
\begin{equation}\label{eq:measure-thresholds}
r_m=m+\lfloor\log_2J(m)\rfloor+1,
\qquad M_m=2m r_m+1.
\end{equation}
To describe the measures, let $H\leq G$ be closed, put $Q_H=H/H^\circ$, and write $\pi_H:H\to Q_H$ for the quotient map. With $m_H$ denoting normalized Haar measure on $H$, set
\[
E_{H,n}=\{(h_1,\ldots,h_n)\in H^n:
\langle\pi_H(h_1),\ldots,\pi_H(h_n)\rangle=Q_H\}.
\]
Whenever $E_{H,n}\neq\varnothing$, define a probability measure on $G^n$ by
\begin{equation}\label{eq:component-haar-measure}
\nu_{H,n}
=\frac{m_H^{\otimes n}|_{E_{H,n}}}
       {m_H^{\otimes n}(E_{H,n})}.
\end{equation}
Thus $\nu_{H,n}$ is product Haar measure conditioned on generating the component group. For connected $H$ it is ordinary product Haar measure on $H^n$; for finite $H$ it is uniform measure on the generating $n$-tuples of $H$.

\begin{revision}
There is also a classification under a qualitative redundancy hypothesis, without a bound involving $m$. A dense representation into $H$ is \emph{two-redundant} if some free factor of rank at most $n-2$ has dense image. For an $\Aut(F_n)$-orbit $\mathcal O\subset\Epi(F_n,Q_H)$, let $\nu_{H,\mathcal O}$ be product Haar measure on $H^n$ conditioned on $(\pi_H(h_1),\ldots,\pi_H(h_n))\in\mathcal O$.

\Needspace{10\baselineskip}
\begin{theorem}[Two-redundancy, measures and orbits]\label{thm:two-redundant-applications}
Let $H$ be a compact Lie group. If an $\Aut(F_n)$-invariant ergodic Borel probability on $H^n$ gives positive measure to two-redundant dense representations, then it equals $\nu_{H,\mathcal O}$ for one orbit $\mathcal O\subset\Epi(F_n,Q_H)$. For connected $H$, this is product Haar measure.

If $H$ is connected and $n\geq4$, every two-redundant dense representation has dense $\Aut(F_n)$-orbit in $H^n$.
\end{theorem}

The quantitative bounds below guarantee two-redundancy and transitivity on the generating tuples of the component group.
\end{revision}

\Needspace{10\baselineskip}
\begin{theorem}[Classification of invariant measures]\label{thm:compact-measure-classification}
Let $G\leq\UU(m)$ be \rev{compact, with no connectedness assumption}, and suppose
\[
n\geq M_m
=2m\bigl(m+\lfloor\log_2J(m)\rfloor+1\bigr)+1.
\]
Every $\Aut(F_n)$-invariant ergodic Borel probability measure on $\Hom(F_n,G)$ equals $\nu_{H,n}$ for a unique closed subgroup $H\leq G$. Conversely, for every closed $H\leq G$, the measure $\nu_{H,n}$ is defined, is invariant and ergodic, and gives full measure to $\Epi(F_n,H)$.
\end{theorem}

In particular, every invariant Borel probability measure is a mixture of the measures in Theorem~\ref{thm:compact-measure-classification}. \rev{The subgroup $H$ need not be connected, even when $G$ is connected.} The conditioning in \eqref{eq:component-haar-measure} is therefore part of the classification.

There is a smaller bound when the measure gives full measure to dense representations into $G$. Let $\ell(G)$ be the maximal length of a strict chain of closed connected subgroups from $1$ to $G$. In particular, $\ell(\UU(m))=2m-1$ and $\ell(\SU(m))=2m-2$ for $m\geq2$; see~\cite{BLScompact}.

\Needspace{11\baselineskip}
\begin{corollary}[\rev{Haar uniqueness and dense orbits}]\label{cor:compact-dense-measures}
Let $G\leq\UU(m)$ be compact, connected and nontrivial. If
\[
n\geq\ell(G)\bigl(m+\lfloor\log_2J(m)\rfloor+1\bigr)+2,
\]
then $m_G^{\otimes n}$ is the unique $\Aut(F_n)$-invariant Borel probability measure on $G^n$ giving full measure to $\Epi(F_n,G)$. \rev{Moreover, every dense representation has dense $\Aut(F_n)$-orbit in $G^n$.} In particular, for $G=\SU(m)$, $m\geq2$, it suffices that
\[
n\geq(2m-2)\bigl(m+\lfloor\log_2J(m)\rfloor+1\bigr)+2.
\]
\end{corollary}

The general classification bound displayed in \cite[Section~5.2.5]{CDB} is
\[
n\geq2m\bigl(1+\tfrac52m+\log_2J(m)\bigr)+3.
\]
The bound $M_m$ lowers this threshold by at least $3m^2+2$ generators after rounding to integers. For $m\geq71$, Collins's formula $J(m)=(m+1)!$~\cite{Collins} makes \eqref{eq:measure-thresholds} explicit; the simpler sufficient bound
\[
n\geq\left\lceil2m^2\log_2m\right\rceil
\]
also yields the full classification, as shown in \cite[Section~9]{PrimitiveSpectra}. These are sufficient quantitative bounds; the general uniqueness conjecture for connected compact targets at every $n\geq3$ is not proved by these results.

\begin{revision}
\subsection{Compact redundancy at rank three}\label{sec:compact-spectral-applications}

Here closures are taken in the compact topology, and redundancy means that some proper free factor has dense image.

For $g\in\UU(d)$ with eigenvalues $\lambda_1,\ldots,\lambda_d$, put $E(g)=\langle\lambda_1,\ldots,\lambda_d\rangle\leq S^1$. The element $g$ is \emph{neat} if $E(g)$ is torsion-free. This is equivalent to connectedness of its cyclic compact closure. The following statements are \cite[Theorem~10.4 and Corollary~10.6]{PrimitiveSpectra}.

\Needspace{12\baselineskip}
\begin{theorem}[A maximal-torus criterion for redundancy]\label{thm:compact-spectral-applications}
Let $G\leq\UU(d)$ be compact and connected, and let $\rho:F_3=\langle x,y,z\rangle\to G$ have dense image. Suppose that
\[
T=\overline{\langle\rho(x)\rangle}\text{ is a maximal torus of }G,
\qquad
H=\overline{\langle\rho(y),\rho(z)\rangle}\text{ is connected}.
\]
Then some primitive $p\in\langle y,z\rangle$ satisfies $\overline{\langle\rho(x),\rho(p)\rangle}=G$. Thus $\rho$ is redundant and Zariski redundant. Sufficient conditions on the three basis images are
\[
E(\rho(y)),E(\rho(z))\text{ torsion-free},
\qquad \rank_\Z E(\rho(x))=\rank G.
\]
There is no bound on the rank or dimension of $G$.
\end{theorem}

Connected cyclic closures alone do not imply redundancy; the full-rank condition on $E(\rho(x))$ is a separate hypothesis. The next consequence uses Theorem~\ref{thm:main} through the proper-factor restriction in Theorem~\ref{thm:primitive-factors}.

\Needspace{6\baselineskip}
\begin{corollary}[Finite primitive spectra with neat generators]\label{cor:compact-neat-applications}
Let $n\geq3$ and let $\rho:F_n\to\UU(d)$ have only finitely many characteristic polynomials on primitive elements. If the images of some free basis are neat, then $\rho$ is trivial.
\end{corollary}
\end{revision}

\begin{revision}
\section{Infinite primitive spectra and nonlinearity}\label{sec:aut-nonlinearity}

Theorem~\ref{thm:main} has the following consequence for faithful representations of free groups. It implies the nonlinearity theorem of Formanek and Procesi~\cite{FormanekProcesi}. Throughout this section, $P_n$ denotes the set of primitive elements of $F_n$, and for a representation $\rho:F_n\to\GL_d(k)$ we write
\[
\mathcal S(\rho)=\{\chi_{\rho(p)}:p\in P_n\}
\]
for its primitive characteristic-polynomial spectrum.

\Needspace{10\baselineskip}
\begin{theorem}[Infinite primitive spectrum]\label{thm:infinite-primitive-spectrum}
Let $n\geq3$, let $k$ be any field, and let $\rho:F_n\to\GL_d(k)$ be faithful. Then $\mathcal S(\rho)$ is infinite.

More generally, if a representation $\rho:F_n\to\GL_d(k)$ has finite primitive spectrum, then $\rho(A)$ is virtually solvable for every proper free factor $A<F_n$.
\end{theorem}

\begin{proof}
Suppose that $\mathcal S(\rho)$ is finite. Let $A<F_n$ be a proper free factor, extend a free basis of $A$ to one of $F_n$, and let $c$ be a complementary basis element. Every element $ca$, $a\in A$, is primitive: the map sending $c$ to $ca$ and fixing the other basis elements is an automorphism. Thus the coset $\rho(c)\rho(A)$ has only finitely many characteristic polynomials. Corollary~\ref{cor:finite-spectrum-coset}, which follows from Theorem~\ref{thm:main}, implies that $\rho(A)$ is virtually solvable.

If $\rho$ is faithful, choose $A$ of rank $n-1\geq2$. Then $\rho(A)\cong F_{n-1}$ is not virtually solvable, a contradiction.
\end{proof}

For $g\in F_n$, write $\iota_g$ for the inner automorphism $v\mapsto gvg^{-1}$. Thus $\iota_g\iota_h=\iota_{gh}$, and for every $\alpha\in\Aut(F_n)$,
\begin{equation}\label{eq:inner-equivariance}
\alpha\iota_g\alpha^{-1}=\iota_{\alpha(g)}.
\end{equation}

\Needspace{9\baselineskip}
\begin{corollary}[Formanek--Procesi]\label{thm:aut-nonlinearity}
Let $n\geq3$ and let $k$ be any field. Then $\Aut(F_n)$ admits no faithful finite-dimensional representation over $k$. Moreover, for every representation $R:\Aut(F_n)\to\GL_d(k)$ and every proper free factor $A<F_n$, the group $R(\iota_A)$ is virtually solvable, where $\iota_A=\{\iota_a:a\in A\}$.
\end{corollary}

\begin{proof}
Set $\rho(w)=R(\iota_w)$. Since $\Aut(F_n)$ acts transitively on $P_n$, identity~\eqref{eq:inner-equivariance} shows that all $\rho(p)$, $p\in P_n$, are conjugate. Hence $\mathcal S(\rho)$ is a singleton. By Theorem~\ref{thm:infinite-primitive-spectrum}, $\rho$ is not faithful and $\rho(A)=R(\iota_A)$ is virtually solvable for every proper free factor $A$. Since $F_n$ has trivial center, the map $w\mapsto\iota_w$ is injective. Therefore $R$ is not faithful either.
\end{proof}
\end{revision}

\begin{remark}[Free products with an infinite cyclic factor]\label{rem:free-product-aut}
\rev{The underlying coset argument also applies to any nontrivial group $G$.} Put $Q=G*\langle c\rangle$, let $H_G=\{\iota_g:g\in G\}\leq\Aut(Q)$, and set $x=\iota_c$. The automorphism fixing $G$ pointwise and sending $c$ to $cg$ satisfies
\[
\tau_g x\tau_g^{-1}=x\iota_g\qquad(g\in G).
\]
Thus every finite-dimensional linear representation of $\Aut(Q)$ over any field maps $H_G$ to a virtually solvable group. Free-product normal form shows that $g\mapsto\iota_g$ is injective on $G$: a nonidentity element of $G$ does not commute with $c$. In particular, if $G$ is not virtually solvable, then $\Aut(G*\Z)$ has no faithful finite-dimensional representation over any field.
\end{remark}

The classical Formanek--Procesi theorem~\cite{FormanekProcesi} gives a nilpotent-by-abelian-by-finite restriction on proper-free-factor images. The following consequence strengthens that restriction for complex representations of $\Aut(F_n)$ when \rev{$n\geq5$}: the entire identity component of each proper-factor closure is unipotent. In particular, it eliminates the torus quotient permitted by the classical conclusion. The additional input is Proposition~\ref{prop:characteristic-solvable}, whose proof in~\cite{PrimitiveSpectra} uses property~(T).

\Needspace{12\baselineskip}
\begin{corollary}[A strengthening of Formanek--Procesi]\label{cor:aut-proper-unipotent}
Let \rev{$n\geq5$}, let $R:\Aut(F_n)\to\GL_d(\C)$ be a representation, and let $A<F_n$ be a proper free factor. Put
\[
H_A=\{R(\iota_a):a\in A\},\qquad K_A=\Zcl{H_A}.
\]
Then $K_A^\circ$ is unipotent. Thus $H_A$ is virtually unipotent and, in particular, virtually nilpotent. If $K_A$ is connected, then $H_A$ itself is unipotent. There is no restriction on $d$.
\end{corollary}

\begin{proof}
Enlarge $A$ to a free factor $B$ of rank $n-1$, and write $F_n=B*\langle c\rangle$. \rev{Corollary~\ref{thm:aut-nonlinearity} shows that the representation $\rho:B\to\GL_d(\C)$, $\rho(b)=R(\iota_b)$, has virtually solvable image.} Every automorphism $\alpha$ of $B$ extends to an automorphism $\widehat\alpha$ of $F_n$ fixing $c$, and
\[
\rho(\alpha(b))=R(\widehat\alpha)\rho(b)R(\widehat\alpha)^{-1}.
\]
Thus $\rho$ is characteristic. Since \rev{$\rank(B)=n-1\geq4$}, Proposition~\ref{prop:characteristic-solvable} makes $\rho(B)$ virtually unipotent. The same holds for its subgroup $H_A$, giving the stated conclusion on $K_A^\circ$.
\end{proof}

\begin{remark}[The finite-index qualification is necessary]\label{rem:finite-image-obstruction}
Actual unipotence of $H_A$ cannot be asserted without an additional hypothesis, even for arbitrarily large $n$. Let $n\geq3$ and let $\Aut(F_n)$ act on the finite set $X=\Hom(F_n,S_3)$ by $\alpha\cdot\rho=\rho\circ\alpha^{-1}$. Let $R$ be the resulting permutation representation on $\C[X]$. For a basis $x_1,\ldots,x_n$, the inner automorphism $\iota_{x_1}$ interchanges the two points
\[
\bigl((12),(23),1,\ldots,1\bigr),\qquad
\bigl((12),(13),1,\ldots,1\bigr)
\]
of $X\simeq S_3^n$. The difference of their basis vectors is therefore an eigenvector of $R(\iota_{x_1})$ with eigenvalue $-1$. In particular, $R(\iota_{x_1})$ is not unipotent, and neither is the image of the proper free factor $\langle x_1,\ldots,x_{n-1}\rangle$ under inner automorphisms. This finite-image example is compatible with virtual unipotence.
\end{remark}

\section*{Acknowledgments}
We acknowledge the assistance of ChatGPT-6 Astra in mathematical discussions and in the preparation and revision of this manuscript.

\clearpage
\begin{appendices}
\section{Examples and sharpness}\label{sec:examples}

All examples in this appendix are over $\C$. They exhibit constant-spectrum cosets with unipotent, toral, and nonnilpotent solvable subgroups. In particular, the virtual-solvability conclusion of Theorem~\ref{thm:main} cannot be strengthened to virtual nilpotence.

\begin{example}[Unitriangular subgroups and varying Jordan type]\label{ex:unitriangular}
Let $U_d$ be the upper unitriangular subgroup of $\GL_d(\C)$, and let
\[
D=\diag(\lambda_1,\ldots,\lambda_d),\qquad \lambda_i\in\C^*.
\]
Every $Du$, with $u\in U_d$, is upper triangular with diagonal entries $\lambda_1,\ldots,\lambda_d$. Hence
\[
\chi_{Du}(z)=\prod_{i=1}^d(z-\lambda_i).
\]
If the $\lambda_i$ are pairwise distinct, every element of $DU_d$ is diagonalizable and conjugate to $D$. For $d\geq3$, the subgroup $U_d$ is nonabelian; its nilpotency class is $d-1$. Thus even a coset in a single regular semisimple conjugacy class can have a subgroup of arbitrarily large nilpotency class.

At the other extreme, take $D=I$. The coset $U_d$ still has constant characteristic polynomial, but it meets every unipotent conjugacy class: the Jordan matrix associated with any partition of $d$ lies in $U_d$. This illustrates the distinction between a characteristic-polynomial fiber and a single conjugacy class.
\end{example}

\begin{example}[A torus of dimension $d-1$]\label{ex:cyclic-torus}
Let $d\geq2$, let
\[
T=\{\diag(t_1,\ldots,t_d):t_i\in\C^*,\ t_1\cdots t_d=1\},
\]
and let $P$ be the permutation matrix defined by $Pe_i=e_{i+1}$, with indices taken modulo $d$. For $t=\diag(t_1,\ldots,t_d)\in T$, the matrix $Pt$ acts by $e_i\mapsto t_i e_{i+1}$. Thus $(Pt)^d=I$, and $e_1$ is a cyclic vector. It follows that
\[
\chi_{Pt}(z)=z^d-1.
\]
The roots are distinct, so the entire coset $PT$ is conjugate to $P$. In particular, the subgroup in Theorem~\ref{thm:main} need not be unipotent.

The torus dimension $d-1$ is maximal: constant characteristic polynomial on $xH$ implies $\det(h)=1$ for every $h\in H$, and every torus in $\SL_d(\C)$ has dimension at most $d-1$.
\end{example}

\begin{example}[A torus coset of regular unipotent matrices]\label{ex:unipotent-torus}
This is a version of Larsen's example in \cite[Example~9.11]{Gelander}. Set
\[
J=\begin{pmatrix}0&1&0\\-3&0&1\\-8&0&3\end{pmatrix},
\qquad
T=\{D(t)=\diag(t,t^{-1},1):t\in\C^*\}.
\]
A direct determinant calculation gives
\[
\chi_{JD(t)}(z)
=\det\begin{pmatrix}z&-t^{-1}&0\\3t&z&-1\\8t&0&z-3\end{pmatrix}
=z^3-3z^2+3z-1=(z-1)^3.
\]
In fact, with $S(t)=\diag(1,t,t)$ we have the explicit identity
\begin{equation}\label{eq:unipotent-torus-conjugacy}
JD(t)=S(t)JS(t)^{-1}.
\end{equation}
Since $(J-I)^2\neq0$, the matrix $J$ is regular unipotent. Thus $JT$ lies in a single regular unipotent conjugacy class in $\SL_3(\C)$, although $T$ is a nontrivial torus. The conjugating matrix in \eqref{eq:unipotent-torus-conjugacy} can be multiplied by a scalar to have determinant one, without changing the conjugation.

Dimension three is minimal for this phenomenon. Indeed, if a nontrivial torus $T'\leq\GL_2(\C)$ had a unipotent coset $xT'$, then $\det(T')=1$. After simultaneous conjugation, $T'=\{\diag(t,t^{-1}):t\in\C^*\}$. Writing $x=(x_{ij})$, unipotence would give
\[
x_{11}t+x_{22}t^{-1}=2\qquad\text{for every }t\in\C^*,
\]
which is impossible. The one-dimensional case is immediate.
\end{example}

\begin{example}[Solvable but not virtually nilpotent]\label{ex:nonnilpotent}
Write $D_2(t)=\diag(t,t^{-1})$ and put
\[
H=\left\{
\begin{pmatrix}D_2(t)&B\\0&D_2(s)\end{pmatrix}:
t,s\in\C^*,\ B\in M_2(\C)
\right\}\leq\SL_4(\C),
\qquad
w=\begin{pmatrix}0&-1\\1&0\end{pmatrix}.
\]
This is a connected solvable algebraic group, with an abelian normal unipotent subgroup parametrized by $B$ and quotient $(\C^*)^2$. Take $x=\diag(w,2w)$. Since $\chi_{wD_2(t)}(z)=z^2+1$, block triangularity gives
\[
\chi_{xh}(z)=(z^2+1)(z^2+4)\qquad(h\in H).
\]
The four roots are distinct, so $xH$ is contained in the conjugacy class of $x$.

Nevertheless, $H$ is not virtually nilpotent. To see this directly, let
\[
a=\diag(2,1/2,1,1),\qquad n(v)=I_4+vE_{13}.
\]
Both lie in $H$, and $a n(v)a^{-1}=n(2v)$. If $L\leq H$ has finite index, some positive powers $a^k$ and $n(1)^\ell=n(\ell)$ lie in $L$. With the convention $[b,c]=bcb^{-1}c^{-1}$, successive commutators satisfy
\[
\underbrace{[a^k,[a^k,\ldots,[a^k,n(\ell)]]\ldots]}_{r\text{ occurrences of }a^k}
=n\bigl((2^k-1)^r\ell\bigr)\neq I_4
\qquad(r\geq1).
\]
Thus $L$ is not nilpotent. This shows that virtual solvability in Theorem~\ref{thm:main} cannot be strengthened to virtual nilpotence, even for connected $H$ and a coset in one regular semisimple conjugacy class.
\end{example}

\end{appendices}

\end{document}